\documentclass[11pt,a4paper]{paper} 

\usepackage[utf8]{inputenc}
\usepackage[T1]{fontenc}

\usepackage[english]{babel}

\usepackage{a4wide}

\usepackage{amssymb}
\usepackage{amsthm,amsmath}
\usepackage{amsfonts}
\usepackage{latexsym}

\usepackage{url} 
\usepackage[all]{xy} 
\usepackage{enumerate}
\usepackage{color}

\usepackage{graphicx}

\theoremstyle{plain}
\newtheorem{thm}{Theorem}
\newtheorem{lem}{Lemma}[section]

\newtheorem{prop}[lem]{Proposition}

\theoremstyle{definition}
\newtheorem{df}[lem]{Definition}
\newtheorem{rem}[lem]{Remark}
\newtheorem{ex}[lem]{Example}

\newcommand{\bbZ}{\mathbb{Z}}

\newcommand{\bbC}{\mathbb{C}}

\newcommand{\fg}{{\mathfrak{g}}}

\newcommand{\cF}{{\mathcal{F}}}
\newcommand{\cB}{{\cal{B}}}

\newcommand{\cL}{{\cal{L}}}

\begin{document}

\title{Extensions of superalgebras of Krichever-Novikov type}

\date{}

\author{Marie Kreusch \\ D\'epartement de math\'ematique, Universit\'e de Li\`ege,\\
 Grande traverse 12 (B37), B-4000 Li\`ege, Belgique; m.kreusch@ulg.ac.be
}

\maketitle

\begin{abstract}
An explicit construction of central extensions of Lie superalgebras of Krichever-Novikov type is given. In the case of Jordan superalgebras related to the superalgebras of Krichever-Novikov type we calculate a 1-cocycle with coefficients in the dual space.
\end{abstract}

\maketitle

{\bf Key Words}: 
Krichever-Novikov Lie superalgebras, Jordan superalgebras, Lie antialgebras, Gelfand-Fuchs cocycle.

\medskip


\section{Introduction}

In 1987, I. M. Krichever and S. P. Novikov \cite{KN1987}, \cite{KN1988} and \cite{KN1989} introduced and studied a family of Lie algebras generalizing the Virasoro algebra. 
Krichever-Novikov algebras are obtained as central extensions
of the Lie algebras of meromorphic vector fields on a Riemann surface of arbitrary genus $g$
with two marked points. 
M. Schlichenmaier studied the Krichever-Novikov Lie algebras for more than two
marked points \cite{Sch1990}, \cite{Sch1990bis} and \cite{Sch1990bisbis}.
He showed, in particular, the existence of local 2-cocycles 
and central extensions for multiple-point Krichever-Novikov algebras \cite{Sch2003}, 
extending the explicit formula of 2-cocycles due to Krichever and Novikov. 
Deformations on theses algebras were studied in \cite{FS2003} and \cite{FS2005} by A. Fialowski and M. Schlichenmaier. 

The notion of Lie antialgebra was introduced by V. Ovsienko  in
 \cite{Ovs2011}, where the geometric origins were explained. 
It was then shown in \cite{LM2012} that these algebras are particular cases of Jordan superalgebras. 
The most important property of Lie antialgebras is their relationships with Lie superalgebras
see \cite{Ovs2011}, \cite{M2009}, \cite{LM2012} and \cite{LM2011};
different from the classical Kantor-Koecher-Tits construction 
for general Jordan superalgebras. One of the main example of \cite{Ovs2011} is the conformal Lie antialgebra
$\mathcal{AK}(1)$ closely related to the Virasoro algebra 
and the Neveu-Schwarz Lie superalgebra $\mathcal{K}(1)$. In \cite{M2009}, S. Morier-Genoud studied  an other important finite dimensional Lie antialgebra: $\mathcal{K}_3$, called the Kaplansky Jordan superalgebra which is related to $osp(1|2)$.

Lie superalgebras of Krichever-Novikov type, $\mathcal{L}_{g,N}$, and the relation with Jordan superlagebras of Krichever-Novikov type,  $\mathcal{J}_{g,N}$, were studied by S. Leidwanger and S. Morier-Genoud in \cite{LM2011}. 
They found examples of Lie antialgebras generalizing $\mathcal{AK}(1)$, as the same way that $\mathcal{L}_{g,N}$ generalizes $\mathcal{K}(1)$. In the present paper, we study central extensions of
$\mathcal{L}_{g,N}$ and the corresponding 1-cocycles on  $\mathcal{J}_{g,N}$.

Our first theorem is an explicit formula for a local non-trivial 2-cocycle on $\mathcal{L}_{g,N}$.
This formula uses projective connections and is very similar
to the formula of Krichever-Novikov and Schlichenmaier. 
This formula was first studied in the case of two points by P. Bryant in \cite{B1990}.
We prove that the cohomology class of this 2-cocycle is independent of the choice of
the projective connection.
Fixing a theta characteristics and a splitting (and therefore an almost-grading) for the superalgebra $\mathcal{L}_{g,N}$, there is a unique up to equivalence and scaling non-trivial almost-graded central extension.  

Our second theorem is an explicit formula for a local 1-cocycle on $\mathcal{J}_{g,N}$ with coefficient in the dual space.
Recently,  P. Lecomte and V. Ovsienko  introduced a cohomology theory
of Lie antialgebras in~\cite{LO2011}. 
In particular, they discovered two non-trivial cohomology classes
of the conformal Lie antialgebra $\mathcal{AK}(1)$
analogous to the  Gelfand-Fuchs class and to the Godbillon-Vey class. 
The cocycle on $\mathcal{J}_{g,N}$ studied in this paper satisfies similar properties than those found in \cite{LO2011}. 
It is given by a very simple and geometrically natural formula
that explains the geometric nature of Lie antialgebras
associated to Riemann surfaces.

Interesting explicit examples of superalgebras arise in the case of the Riemann sphere
with three marked points.
These examples were thoroughly studied in \cite{Sch1990} and  \cite{LM2011}.
The corresponding Lie superalgebra is denoted by $\mathcal{L}_{0,3}$ and
the  Jordan superalgebra by $\mathcal{J}_{0,3}$. 
These two algebras are closely related since $\mathcal{L}_{0,3}$ 
is the adjoint superalgebra of $\mathcal{J}_{0,3}$. 
Moreover, these algebras contain the conformal algebras:
 $$
  \mathcal{L}_{0,3}\supset \mathcal{K}(1) \supset \mathrm{osp}(1|2)
  \qquad\hbox{and}\qquad 
   \mathcal{J}_{0,3}\supset \mathcal{AK}(1) \supset  \mathcal{K}_3.
   $$
We calculate explicitly the $2$-cocycle on the Lie superalgebra $\mathcal{L}_{0,3}$
that is unique up to isomorphism and vanishes on the Lie subalgebra $\mathrm{osp}(1|2)$.
 This 2-cocycle induces a $1$-cocycle on 
$\mathcal{L}_{0,3}$ with values in its dual space. 
Finally, we give an explicit formula for a $1$-cocycle 
on $\mathcal{J}_{0,3}$ with values in its dual space.

The paper is organized as follows.
In Section 2, we recall some definitions and main results on the Krichever-Novikov Lie algebras. In particular, we consider $2$-cocycles on these Lie algebras and recall some tools that we will use in the computation of  cocycles in the case of the Riemann sphere. 
In Section $3$, we give the basic definitions of Lie superalgebras $\mathcal{L}_{g,N}$ with significant examples : $\mathcal{K}(1)$ and $\mathcal{L}_{0,3}$. We give a local non-trivial 2-cocycle on $\mathcal{L}_{g,N}$ (Theorem \ref{MainFirst}) and 
in particulat on $\mathcal{L}_{0,3}$.
In Section $4$, we recall the basic notions of Lie antialgebras 
with examples and their relation to Lie superalgebras. 
In section $5$,  we construct a 1-cocycle on $\mathcal{L}_{0,3}$ related to the $2$-cocycle found in section $3$. After, we give a local $1$-cocycle on $\mathcal{J}_{g,N}$ (Theorem \ref{MainSecond}) and construct the unique (up to isomorphism) 1-cocycle on $\mathcal{J}_{0,3}$ that vanishes on $\mathcal{K}_3$.

\section{Lie algebras of Krichever-Novikov type}

In \cite{KN1987}, \cite{KN1988} and \cite{KN1989}, 
Krichever and Novikov introduced some generalizations of the well known 
Witt algebra and its central extension called the Virasoro algebra 
(see \cite{S2003} for a global overview of this theory). 
In this section, we recall the definitions and basic results needed for the sequel. 
All the structures in this paper will be considered over the field $\mathbb{C}$.

\subsection{Definition and examples}

Let $M$ be a compact Riemann surface of genus $g$ 
(i.e., a smooth projective curve over $\mathbb{C}$).
Consider the union of two sets of ordered disjoint points called \textit{punctures}
\[
A = \underbrace{( P_1, \ldots , P_K)}_{:=I} \cup \underbrace{( Q_1, \ldots, Q_{N-K})}_{:=O}
\]
where $N,K \in \mathbb{N} \backslash \{0\}$ with $N\geq 2$ and $1 \leq K < N$. 
We call $I$, the set of \textit{in-points}, and $ O$ the set of \textit{out-points}.
Denote by $ \mathfrak{a}_{g,N}$ the associative algebra of meromorphic functions on $M$ 
which are holomorphic outside of $A$. 

The Krichever-Novikov algebra $\mathfrak{g}_{g,N}$ 
is the Lie algebra of meromorphic vector fields on $M$ 
which are holomorphic outside of $A$ 
($\mathfrak{g}_{g,N}$ is equipped with the usual Lie bracket of vector fields).
We will use the same symbol for the vector field and its local representation so that the Lie bracket is 
\[
\left[e(z)\frac{d}{dz},f(z) \frac{d}{dz} \right] = \left( e(z) f'(z) - f(z) e'(z) \right) \frac{d}{dz}.
\] 

If $g=0$, one considers the Riemann sphere $\mathbb{CP}^1$ with punctures.
The moduli space ${\mathcal M}_{0,N}$ is of dimension $N-3$.
This means that, for $N\leq3$, the points can be chosen in an arbitrary way
providing isomorphic algebraic structures.
Note also that $\mathbb{CP}^1$ can be equipped
with a ``quasi-global'' coordinate $z$.

In the case $g=0$ and $N=2$,
one can take $I= \{0\}$ and $O= \{\infty\}$ and the Krichever-Novikov algebra $\mathfrak{g}_{0,2}$ 
is nothing but the Witt algebra. 
It admits a basis $\{e_n = z^{n+1}\frac{d}{dz} : n \in \mathbb{Z}\}$ satisfying the relations: 
\[
\left[ e_n, e_m \right] = (m-n)e_{n+m}.
\]
The (unique) non-trivial central extension of the Witt algebra is well-known, 
it is called the Virasoro algebra.
This algebra has a basis $\{e_n = z^{n+1}\frac{d}{dz} : n \in \mathbb{Z}\}$ 
together with the central element $c$, such that 
\[
[e_n, e_m] = (m-n)e_{n+m} + \frac{1}{12}(m^3-m)\delta_{n,-m} c  \hspace{0.2cm} , \hspace{0.5cm} [e_n, c] = 0.
\] 
The algebra of functions $\mathfrak{a}_{0,2}$
 is the algebra of Laurent polynomials $\mathbb{C} [z,z^{-1}]$. 

Another simple example considered in \cite{Sch1993} and further in \cite{FS2003}  is the case $g=0$ and $N=3$.
The marked points are then chosen as follows: $I = \{\alpha, - \alpha\}$ and $O= \{ \infty\}$, 
where $\alpha \in \mathbb{C}\backslash \{0\}$.
The Lie algebra $\mathfrak{g}_{0,3}$ is
spanned by the following vector fields, for all $k \in \mathbb{Z}$ :
\begin{equation}
\label{Sch}
V_{2k}(z) = z (z-\alpha)^k (z+ \alpha)^k \frac{d}{dz} \hspace{0.2cm}, 
\qquad  V_{2k+1}(z) =  (z-\alpha)^{k+1} (z+ \alpha)^{k+1} \frac{d}{dz}.
\end{equation}

\subsection{Construction of a $2$-cocycle on $\mathfrak{g}_{g,N}$}

Let us recall the construction of a $2$-cocycle on $\mathfrak{g}_{g,2}$
due to Krichever and Novikov \cite{KN1987} and \cite{KN1988}
and further generalized on $\mathfrak{g}_{g,N}$ by Schlichenmaier \cite{Sch2003}.

Given a Riemann surface and $(U_{\alpha}, z_{\alpha})_{\alpha \in J}$ a covering 
by holomorphic coordinates with transition functions $z_{\beta} = g_{\beta\alpha}(z_{\alpha})$,
a \emph{projective connection} is a system of functions 
$R= (R_{\alpha}(z_{\alpha}))_{\alpha \in J}$ transforming as 
\[
R_{\beta}(z_{\beta}) .  \left(g'_{\beta\alpha}\right)^2 = R_{\alpha}(z_{\alpha}) + S (g_{\beta\alpha})  , 
\]
where
\[
S(g)= \frac{g'''}{g'}-\frac{3}{2} \left( \frac{g''}{g'}\right)^2
\]
is the \emph{Schwarzian derivative} (see \cite{OT2005}) and where $'$ denotes differentiation 
with respect to the coordinate~$z_\alpha.$
It is a classical result that every Riemann surfaces admits a holomorphic projective connection; see \cite{G1966} or \cite{HS1996} (p. 202).

Given a splitting $A = I \cup O$, the algebras $\mathfrak{a}_{g,N}$ and $\mathfrak{g}_{g,N}$ are almost-graded and the modules of tensor densities $\mathcal{F}_{\lambda}$, $\lambda \in \mathbb{Z}$ are almost-graded modules over them; see \cite{Sch2003} pp. 58--61 for more details. 
Moreover, for every \emph{separating cycle} $\mathcal{C}$ (separates the points in $I$ from the points in $O$), with respect to a chosen projective connection $R$, there is a 2-cocycle on 
$\mathfrak{g}_{g,N}$ defined by 
\begin{equation}
\label{CocKN}
\gamma_{\mathcal {C},R} \left(e(z) \frac{d}{dz}, f(z)\frac{d}{dz}\right) = 
\frac{1}{2i\pi}  \int_{\mathcal{C}}  \left(\frac{1}{2} (e^{'''}f-ef^{'''}) - R (e'f-ef') \right) dz.
\end{equation}
Note that, in the case $g=0$,  one can consider $R \equiv 0$.
The 2-cocycle (\ref{CocKN}) can be understood as a generalization
of the famous Gelfand-Fuchs cocycle; see \cite{Fuks1986}.

M. Schlichenmaier \cite{Sch2003} proved that the cohomology class of $\gamma_{\mathcal{C},R}$ 
does not depend on the choice of the connection $R$ (this is a simple calculation). 
He also proved that this cocycle is cohomologically non-trivial, local
(i.e., preserves the almost grading) and that every local cocycle on
$\fg_{g,N}$ is either a coboundary or a scalar multiple of $\gamma_{\mathcal {C},R}$ with $R$ a meromorphic projective connection which is holomorphic outside $A$.  

If the cocycle is local with respect to the splitting, then $\mathcal{C}$ can be taken as a sum of (small) circles around the points in $I$: $\mathcal{C} = \sum^{K}_{i} \mathcal{C}_i$. The integral in (\ref{CocKN}) can be written in the complex analytic setting in terms of the residues.
The Riemann sphere ($g$=0) can be viewed as the structure of the extended complex plane $\widehat{\bbC}$; see \cite{M}.
In the next section, we calculate the residue at $\infty$ and considering the function 
$f_{1/z} : z \mapsto f(\frac{1}{z})$, one has:
\[
Res_{\infty} (f) = - Res_0 \left(\frac{f_{1/z}}{z^2} \right), 
\] 
and moreover, if $z_0 \in \bbC$ is a pole of $f$, of order $p\in \mathbb{N} \backslash \{0\}$, then 
\[
Res_{z_0} f = \frac{1}{(p-1)!}   \lim_{z \to z_0} D^{p-1}\left((z-z_0)^p f(z)\right).
\]

\section{Lie superalgebras of K-N type and their central extensions}

In this section, we recall the notion of Lie superalgebras of Krichever-Novikov type 
$\mathcal{L}_{g,N}$.
We show the existence of a local non-trivial 2-cocycle on $\mathcal{L}_{g,N}$ satisfying similar
 properties to those of the cocycle (\ref{CocKN}).
We consider, in particular, the case $g=0$ and $N=3$, namely, 
the Lie superalgebra~$\mathcal{L}_{0,3}$
and compute the corresponding 2-cocycle explicitly.

\subsection{Definition and examples of Lie superalgebras}

A \emph{Lie superalgebra} is a  $\mathbb{Z}_2$-graded vector space, $\mathcal{L}=   \mathcal{L}_0 \oplus \mathcal{L}_1$, equipped with a bilinear product (Lie bracket), such that 
\begin{enumerate}
\item[(LS1)] \mbox{super skewsymmetry :}  
\hspace{0.5cm}  $ [x,y] = - (-1)^{\bar{x} \bar{y}} [y,x] $
\item[(LS2)] \mbox{super Jacobi identity :} 
\hspace{0.5cm}$ (-1)^{\bar{x}\bar{z}}\left[x,[y,z] \right] 
+ (-1)^{\bar{y}\bar{x}} \left[y,[z,x]\right] + (-1)^{\bar{z}\bar{y}} \left[z,[x,y] \right]= 0 $
\end{enumerate}
for all homogeneous elements $x,y,z$ in $\mathcal{L}.$ 
The subspace  $\mathcal{L}_0$ is the space of \emph{even elements} and the subspace $\mathcal{L}_1$
 is that of \emph{odd elements}.
  The \emph{degree} of a homogeneous element $x$ is denoted by $\bar{x}$, 
  i.e. $\bar{x}= i$ for $x \in \mathcal{L}_i$.
 
\begin{ex} \label{exampleK1}
The {\it conformal Lie superalgebra}
$\mathcal{K}(1)$ is an infinite-dimensional Lie superalgebra with basis
$\{e_n ,\hspace{0.1cm} n \in \mathbb{Z}\}$ of the even part and $\{b_i, \hspace{0.1cm} i \in \mathbb{Z}+ \frac{1}{2} \}$
of the odd part satisfying the relations: 
\[
\left\{
\begin{array}{l}
	\left[ e_n, e_m \right] = \left( m - n \right) e_{n+m} \\[6pt]
	\left[ e_n,  b_i \right]  =  \left(  i -\frac{n}{2} \right)  b_{i+n} \\[6pt]
	\left[ b_i , b_j \right]  =  e_{i+j} .
\end{array}
\right.
\] 
The even part of $\mathcal{K}(1)$ coincides thus with the Witt algebra $\mathfrak{g}_{0,2}$. The  elements 
$\{
b_{-\frac{1}{2}},
b_{\frac{1}{2}},
e_{-1},
e_0,
e_1
\}$
span the classical simple Lie superalgebra $\mathrm{osp}(1|2)$.
\end{ex}

\subsection{The Lie superalgebras $\mathcal{L}_{g,N}$}

The Lie superalgebras of Krichever-Novikov type were studied in \cite{LM2011}. 
Let us briefly recall the main definition.

Let $K$ be the canonical (i.e. holomorphic cotangent) line bundle of $M$ and consider the bundle $K^{\otimes \lambda}$, $\lambda \in \mathbb{Z}$ where we assert that  $K^{\otimes 0}$ is the trivial bundle and $K^{\otimes -1} = K^{*}$ . A square root of $K$ (also called a \emph{theta characteristics}) is a line bundle $S$ such that $S^{\otimes 2}= K.$ On a Riemann surface of genus $g$, since the degree of $K$ is even, the number of different square roots equals $2^{2g}.$ Despite for $g=0$, the theta characteristics is not unique. So, let fix one on $M$ for the rest of the paper, now we can consider the bundle $K^{\otimes \lambda}$ where $\lambda \in \mathbb{Z} \cup \frac{1}{2} + \mathbb{Z}$. Denote $\mathcal{F}_{\lambda}$ the (infinite-dimensional) vector space of global meromorphic sections of $K^{\otimes \lambda}$ which are holomorphic on $M \backslash A$, also called space of tensor densities of weight $\lambda$ in \cite{LM2011}. The space $\mathcal{F}=\oplus_{\lambda} \mathcal{F}_{\lambda}$ is a Poisson algebra 
with the following bilinear operations (given in local coordinates): 
\begin{align*}
\bullet \hspace{0.2cm} &  :\hspace{0.2cm}  
\cF_{\lambda} \times \cF_{\mu} \longrightarrow \cF_{\lambda+\mu} : 
& \hspace{0.2cm}    (e(z) dz^{\lambda} , f(z) dz^{\mu}  )   \hspace{0.2cm} 
& \longmapsto \hspace{0.2cm}  e(z) f(z) dz^{\lambda+\mu} \\
\{, \} \hspace{0.2cm} &  :
\hspace{0.2cm}  
\cF_{\lambda} \times \cF_{\mu} \longrightarrow \cF_{\lambda+\mu+1} : 
& \hspace{0.2cm}    (e(z) dz^{\lambda} , f(z) dz^{\mu} )   \hspace{0.2cm}  
& \longmapsto \hspace{0.2cm}  \left(\mu e'(z) f(z) - \lambda e(z) f'(z)\right) dz^{\lambda+\mu+1}.
\end{align*}

We have the Lie algebra isomorphism $\mathfrak{g}_{g,N} \cong \cF_{-1} $,
and the natural action of the Lie algebra   $\mathfrak{g}_{g,N}$ on $\cF_{-1/2}$ is given by the above
Poisson bracket. 

\begin{df}
The \emph{Lie superalgebra of Krichever-Novikov}, denoted by $\cL_{g,N}$, is the vector space 
$\left(\cL_{g,N})_0 \oplus (\cL_{g,N}\right)_1=  \mathfrak{g}_{g,N}\oplus \cF_{-1/2}$ 
with the Lie bracket defined by
\begin{align*}
[e(z) (dz)^{-1}, f(z) (dz)^{-1}] & =  \{ e(z) (dz)^{-1}, f(z) (dz)^{-1} \}   \\[6pt]
[e(z) (dz)^{-1}, \psi(z) (dz)^{-1/2}] & =  \{ e(z) (dz)^{-1} , \psi(z) (dz)^{-1/2} \}  \\[6pt]
[\varphi(z) (dz)^{-1/2}, \psi(z) (dz)^{-1/2}] & = \frac{1}{2}  \varphi(z) (dz)^{-1/2} \bullet \psi(z) (dz)^{-1/2} .
\end{align*}
\end{df} 

\noindent
The axioms of Lie superalgebras can be easily checked.

More precisely, we can write in coordinates:
\[
\left\{
\begin{array}{l}
	\left[ e (dz)^{-1}, f (dz)^{-1} \right] = \left(-e' f + e f' \right) (dz)^{-1}  \\[6pt]
	\left[ e (dz)^{-1}, \psi (dz)^{-1/2} \right]  = \left( -\frac{1}{2}e' \psi + e \psi ' \right) (dz)^{-1/2}  \\[6pt]
	\left[ \varphi (dz)^{-1/2}, \psi (dz)^{-1/2} \right]   = \frac{1}{2}\, \varphi \psi (dz)^{-1} .
\end{array}
\right.
\]

\begin{ex}
a)
In the case of two marked points $  A = \{0 \} \cup \{\infty \}$ on the Riemann sphere, 
we can identify $\mathcal{L}_{0,2}$ with $\mathcal{K}(1)$; see Example \ref{exampleK1}. We have the following identification: 
\[
e_n = z^{n+1} (dz)^{-1}, \hspace{0.5cm} b_i = \sqrt{2} \hspace{0.1cm} z^{i + 1/2} (dz)^{-1/2}.
\]

b)
Consider the Lie superalgebra $\cL_{0,3}$
associated with the Riemann sphere with three punctures
$ A = \{- \alpha, \alpha \} \cup \{\infty \}$,
where $\alpha  \in \bbC \backslash \{0\}$.
According to \cite{Sch1993}; see also \cite{FS2003}, the even part of $\mathcal{L}_{0,3}$, namely $\fg_{0,3}$, has the basis (\ref{Sch}).
The odd part, $\cF_{-1/2}$, according to \cite{LM2011}, has the basis
\begin{equation}
\label{SchBis}
\varphi_{2k+\frac{1}{2}}(z) = \sqrt{2} z (z-\alpha)^k (z+ \alpha)^k  \,dz^{-1/2} , 
\qquad  \varphi_{2k-\frac{1}{2}}(z) = \sqrt{2}  (z-\alpha)^k (z+ \alpha)^k  \,dz^{-1/2},
\end{equation}
 for all $k \in \mathbb{Z}$.
The Lie bracket of this algebra is given in \cite{LM2011}; it is shown 
in particular that the sub-superalgebra 
$\mathcal{L}^-_{0,3} = \left\langle{}V_n  \hspace{0.05cm} : 
\hspace{0.05cm}n \leq 0 \hspace{0.05cm} ; \hspace{0.1cm}\varphi_i : 
\hspace{0.05cm}  i \leq \frac{1}{2}\right\rangle$ of $\mathcal{L}_{0,3}$ 
is isomorphic to $\mathcal{K}(1)$. 
\end{ex}

\subsection{A non-trivial 2-cocycle on $\mathcal{L}_{g,N}$}

In this section, we show that every Lie superalgebra $\mathcal{L}_{g,N}$
has a non-trivial central extension according to the splitting (recall that a theta characteristics is fixed).
To this end, we construct a non-trivial 2-cocycle quite similar to (\ref{CocKN}).
Note that this probles has already been
studied by Bryant in \cite{B1990} for the case $N=2$.

Recall that a \emph{$2$-cocycle} on a Lie superalgebra $\cL$ is an even 
bilinear function $c : \cL \times \cL \longrightarrow \bbC$
satisfying the following conditions: 
\begin{enumerate}
\item [(C1)]
\mbox {super skewsymmetry:}  \hspace{0.5cm}  
$ c(u,v) = - (-1)^{\bar{u} \bar{v}} c (v,u) $
\item [(C2)]
\mbox{super Jacobi identity:} \hspace{0.5cm}
$ c \left( u,[v, w] \right) =   
c \left( [u,v], w \right) + (-1)^{\bar{u}\bar{v}} c \left(v, [u,w]\right)   $
\end{enumerate}
for every homogeneous elements $u,v, w $ $\in \cL$.
As in the usual Lie case, a 2-cocycle defined a central extension of $\cL$.
A 2-cocycle is called {\it trivial}, or a {\it coboundary}
if it is of the form $c(u,v)=f([u,v])$, where $f$ is a linear function on $\cL$.
Otherwise, $c$ is called \textit{non-trivial}.
The space of all 2-cocycles is denoted by $Z^2(\cL)$ and the space of 2-coboundaries
by $B^2(\cL)$, the quotient-space $H^2(\cL)=Z^2(\cL)/B^2(\cL)$ 
is called the second cohomology space of $\cL$.
This space classifies non-trivial central extensions of $\cL$.

The first result of this paper is the following.

\begin{thm}
\label{MainFirst}
(i)
The even bilinear map 
$c:\mathcal{L}_{g,N}\times\mathcal{L}_{g,N}\to\mathbb{C}$
given by
\begin{equation}
\label{CocLie}
\begin{array}{rcl}
c \left(e(z) \frac{d}{dz}, f(z)\frac{d}{dz}\right) 
&=&
\displaystyle
 \frac{- 1}{2i\pi}  \int_{\mathcal{C}}  \frac{1}{2} (e^{'''}f-ef^{'''}) - R (e'f-ef')  dz,\\[16pt]
 c \left(\varphi(z)dz^{-1/2}, \psi(z)dz^{-1/2}\right) 
 &=&
  \displaystyle
\frac{1}{2i\pi}  \int_{\mathcal{C}}\frac{1}{2}    \left( \varphi^{''} \psi + \varphi \psi^{''} \right) 
   -\frac{1}{2}  R\, \varphi \psi dz,\\[16pt]
c \left(e(z) \frac{d}{dz}, \psi(z)dz^{-1/2}\right) 
 &=&
  \displaystyle
0
\end{array}
\end{equation}
where $\mathcal{C}$ is a separating cycle, is a local non-trivial 2-cocycle.

(ii)
The expressions in (\ref{CocLie}) does not depend on the choice of the projective connection.
\end{thm}

\begin{proof}
Part (i).
To show that the above integral is well defined, one notices that,
after coordinate changes $z_{\beta}= g_{\beta, \alpha} (z_\alpha)$, 
the expressions in the both parts of (\ref{CocLie}) are transformed as 1-forms.  
Furthermore, the expression (\ref{CocLie}) is globally defined. 
The cocycle condition is then straightforward. 
Since $c$ is cohomologically non-trivial on the even part (see \cite{FS2003}, p933) it is also the case on $\mathcal{L}_{g,N}$. 

Part (ii).
In the even case, the result is due to Schlichenmaier; see \cite{Sch2003}, p64.
Let $R'$ be a different projective connection,
then $R-R'$ is a well-defined quadratic differential.
The 2-cocycle $ c-c'$
depends only on the Lie bracket of the elements, on the odd part, we have :  
\[
c_{R}(\varphi, \psi) - c'_{R'}(\varphi, \psi) = \frac{1 }{2i \pi} \int_{\mathcal{C}} -\frac{1}{2} (R-R') \varphi \psi dz  =  \frac{1 }{2i \pi} \int_{\mathcal{C}} \{  (R'-R) (dz)^2  \bullet [ \varphi , \psi] (dz)^{-1} \}
\] 
and therefore the above expression is a coboundary. 
\end{proof}

Let us mention that the term ``local'' were introduced by Krichever and Novikov in \cite{KN1987} 
 since it preserves the almost-graded structure.   




\subsection{The case of genus zero}

Let us now assume that $g=0$ and consider the Lie superalgebra $\cL_{0,N}$.
Choose the projective connection $R \equiv 0$ (in the standard flat coordinate $z$)
 adapted to the standard projective structure on $\mathbb{CP}^1$. 
 
An important property of $\cL_{0,N}$ is that it
contains a subalgebra isomorphic to $\mathrm{osp}(1|2)$
that consists in holomorphic
vector fields and $-1/2$-densities.
The Lie superalgebra $\cL_{0,N}$ also contains many copies
of the conformal Lie superalgebra $\mathcal{K}(1)$ consisting in densities
holomorphic outside two points of the set $A$.

In the case $N=2$, there is exactly one splitting on $\mathcal{L}_{0,2}$, hence one almost-grading. So that, we only have one local cocyle, up to a coboundary (and one equivalence class of almost-graded central extension).
The case where $N=3$ is also special since the role of the points in $A$ can be switched by 
${\rm PSL}(2)$-action. Hence, up to isomorphism (but not equivalence) there is only one class of almost-graded central extension.
Moreover, we are interested in the cocycle that vanishes on the Lie subalgebra $\mathrm{osp}(1|2)$, let us now compute the explicit formula for the 2-cocycle (\ref{CocLie}) with $R \equiv 0$ on $\mathcal{L}_{0,3}$.  

\begin{prop}
Up to isomorphism, the $2$-cocycle on the Lie superalgebra $\mathcal{L}_{0,3}$ vanishing on $\mathrm{osp}(1|2)$  is given by

\begin{equation}
\label{Explicit}
\begin{array}{rcl}
c \left( \varphi_{2k+\frac{1}{2}} , \varphi_{2l+\frac{1}{2}}  \right) & = & 0 \\[6pt]
c \left( \varphi_{2k-\frac{1}{2}} , \varphi_{2l-\frac{1}{2}}  \right)  & = &0 \\[6pt]
c \left( \varphi_{2k+\frac{1}{2}} , \varphi_{2l-\frac{1}{2}}  \right)  & = &   4 k (2k +1) \delta_{k+l,0} + 8 \alpha^2 k (k-1) \delta_{k+l,1} \\[6pt]
c \left( V_{2k} , V_{2l}  \right) &  = & - 2  k (4 k^2 -1 ) \delta_{k+l,0} - 8 \alpha^2 k (k-1) (2k-1) \delta_{k+l,1} \\
& & - 8  \alpha^4 k (k-1)(k-2) \delta_{k+l,2} \\[6pt]
c\left( V_{2k+1} , V_{2l+1}  \right)  &= & - 8   \alpha^2 (k+1) k (k-1)  \delta_{k+l,0} - 4 k (k+1) (2k+1) \delta_{k+l,-1} \\[6pt]
c \left( V_{2k} , V_{2l+1}  \right)  & =& 0,
\end{array}
\end{equation}
for all $k,l $ $\in \mathbb{Z}$.
\end{prop}

\begin{proof}
Let us give the details of the calculation for the even generators, 
the others cases are similar. 
\begin{eqnarray*}
c \left( V_{2k+1} ,  V_{2l + 1 }\right)  & = & -\frac{1}{2i\pi} \int_{\mathcal{C}_{\alpha} \cup \mathcal{C}_{- \alpha}}   \{ ( z^2 - \alpha^2)^{k+1}\}''' (z^2 - \alpha^2)^{l+1}  \\
&&\hspace{1.5cm}-  \{ ( z^2 - \alpha^2)^{l+1}\}''' (z^2 - \alpha^2)^{k+1}  dz \\ 
& = &  \frac{1}{2i\pi} \int_{\mathcal{C}_{\infty} } 6 z \{(k+1)k - (l+1)l \} (z^2 - \alpha^2)^{k+l} \\
& & \hspace{1.5cm}+ 4 z^3 \{(k+1)k(k-1) - (l+1)l (l-1) \} (z^2 - \alpha^2)^{k+l-1}  dz\\ 
& = & - 6  \{(k+1)k - (l+1)l \}   Res_{0} \left( \frac{(1- z^2 \alpha^2)^{k+l} }{z^{2k+2l +3}} \right) \\
& & -  4  \{(k+1)k(k-1) - (l+1)l (l-1) \}  Res_{0} \left( \frac{(1- z^2 \alpha^2)^{k+l-1} }{z^{2k+2l +3}} \right) .
\end{eqnarray*}
Consider the residues.
If $k+l \leq -2 $, then the functions are holomorphics near $0$ and the residues vanish,
 and if $k+l \geq 1 $ they also vanish taking into account the Taylor development. 
Consider the remaining cases:
\begin{eqnarray*}
\mbox{$k+l = 0$}, \hspace{0.3cm}& \mbox{then } &\hspace{0.3cm}  Res_{0} \left( \frac{1}{z^3}  \right) = 0  \hspace{0.3cm} \mbox{and} \hspace{0.3cm}  Res_{0} \left( \frac{(1- z^2\alpha^2)^{-1}}{z^3}  \right) =  \alpha^2   \\
 \mbox{$k+l = - 1$}, \hspace{0.3cm} &\mbox{then }& \hspace{0.3cm}  Res_{0} \left( \frac{(1- z^2\alpha^2)^{-1}}{z}  \right) = 1  \hspace{0.3cm} \mbox{and} \hspace{0.3cm}  Res_{0} \left( \frac{(1- z^2\alpha^2)^{-2}}{z}  \right) = 1.
\end{eqnarray*} 
Finally, we obtain 
\begin{eqnarray*}
c \left( V_{2k+1} ,  V_{2l + 1 }\right)  & =  &-  6  \{(k+1)k - (l+1)l \}    \delta_{k+l,-1} \\
& & -  4  \{(k+1)k(k-1) - (l+1)l (l-1) \}   ( \alpha^2 \delta_{k+l,0} + \delta_{k+l,-1}) \\
& = & -  8 \alpha^2 (k+1) k (k-1) \delta_{k+l,0} -  4 k (k+1) (2k + 1) \delta_{k+l,-1}.
\end{eqnarray*}
Hence the result.
\end{proof}

\section{Jordan superalgebras of K-N type}

In this section, we consider a special type of Jordan superalgebras, introduced by V. Ovsienko in \cite{Ovs2011} under the name of $\hspace{0.01cm}$ ``Lie antialgebras''. Lie antialgebras were studied in \cite{M2009}, \cite{LM2012}. The cohomology theory of these algebras is developed in \cite{LO2011}.
Lie antialgebras of Krichever-Novikov type, $\mathcal{J}_{g,N}$, were introduced in \cite{LM2011}.
We calculate a local non-trivial 1-cocycle on $\mathcal{J}_{g,N}$ with values in the dual space $\mathcal{J}_{g,N}^*$. 
The construction is very similar to that of the 2-cocycle (\ref{CocLie}) and extends the cocycle found in  \cite{LO2011}.

\subsection{Definition and examples of Lie antialgebras}

\begin{df}
A \emph{Lie antialgebra} on $\mathbb{C}$ is a $\mathbb{Z}_2$-graded  
\textit{supercommutative} algebra
 $\mathcal{A}= \mathcal{A}_0 \oplus \mathcal{A}_1$ with a product:   
\[
x \centerdot y = (-1)^{\bar{x}\bar{y}} y \centerdot x,
\]
for all homogeneous elements $x,y\in\mathcal{A}$, satisfying
the following  conditions.
\begin{enumerate}
\item [(i)]
The subalgebra $\mathcal {A}_0$ is \textit{associative}.

\item[(ii)]
For every $a\in \mathcal{A}_1$, the operator of right multiplication by $a$ is an (odd) derivation
of $\mathcal{A}$, i.e.,
\begin{equation}
\label{OdD}
\left(x \centerdot y\right) \centerdot{}a=
\left(x \centerdot{}a\right) \centerdot y+
(-1)^{\bar{x}}\,x \centerdot\left(y\centerdot{}a\right),
\end{equation}
for all homogeneous elements $x,y\in\mathcal{A}$.
\end{enumerate}
\end{df}

\noindent
Note that, in the case where $\mathcal{A}$ is generated by its odd part $\mathcal {A}_1$,
the first axiom of associativity is a corollary of (\ref{OdD}), cf. \cite{Ovs2011},\cite{LM2012}.

\begin{ex}
The first example of finite-dimensional Lie antialgebra is the famous tiny Kaplansky superalgebra, denoted 
by $\mathcal{K}_3$. It was first studied by K. McCrimmon in \cite{M1994} and after by Morier-Genoud in \cite{M2009} under the name of $asl_{2}$. 
The basis is $\{\varepsilon; a, b \}$ where $\varepsilon $ is even and $a,b$ are odds. 
It is defined by the following relations: 
\[
\left\{
\begin{array}{l}
	\varepsilon \centerdot  \varepsilon  = \varepsilon \\[4pt]
	\varepsilon \centerdot  a  = \frac{1}{2} a \hspace{0.5cm}  \varepsilon \centerdot  b  = \frac{1}{2} b \\[4pt]
	a \centerdot b = \frac{1}{2} \varepsilon .
\end{array}
\right.
\] 
The algebra $\mathcal{K}_3$ is an example of exceptional simple Jordan superalgebra.
\end{ex}

\begin{ex}
The second important example is an infinite-dimensional algebra, 
denoted by $\mathcal{AK}(1)$.
Its geometric origins are related to the contact structure on the supercircle $S^{1|1}$. 
The basis of $\mathcal{AK}(1)$ is 
$\{ \varepsilon_n  \hspace{0.02cm}  : \hspace{0.02cm} n \in \mathbb{Z} \}  \oplus \{ a_i \hspace{0.02cm}  : \hspace{0.02cm}   i \in \mathbb{Z} + \frac{1}{2} \}$ and the relations are 
\[
\left\{
\begin{array}{l}
	\varepsilon_n \centerdot  \varepsilon_m  = \varepsilon_{n+m} \\[4pt]
	\varepsilon_n \centerdot  a_i  = \frac{1}{2} a_{i+n} \\[4pt]
	a_i \centerdot a_j = \frac{1}{2}(j-i) \varepsilon_{i+j}.
\end{array}
\right.
\] 
Note that $\langle\varepsilon_0, a_{-1/2}, a_{1/2}\rangle$ as a subalgebra of $\mathcal{AK}(1)$
isomorphic to $\mathcal{K}_3$. 
\end{ex}

\subsection{Relations to Lie superalgebras}

A natural way to link Lie antialgebras and Lie superalgebras is to consider
the Lie superalgebra of derivations
 $Der(\mathcal{A})$. 
In particular, one has :
$
Der(\mathcal{K}_3) \cong \mathrm{osp}(1|2)$
and
$Der(\mathcal{AK}(1)) \cong \mathcal{K}(1)$, cf. \cite{Ovs2011}. 

Another way to associate a Lie superalgebra
$\mathcal{G}_\mathcal{A}$ to an arbitrary Lie antialgebra~$\mathcal{A}$,
called the \textit{adjoint Lie superalgebra},
was elaborated in \cite{Ovs2011},\cite{LM2012}.
 Consider the $\mathbb{Z}_2$-graded space
$
\mathcal{G}_\mathcal{A} = 
\mathcal{G}_0 \oplus \mathcal{G}_1
$
where, 
$ \mathcal{G}_1 = \mathcal{A}_1$ and 
$ \mathcal{G}_0 := (\mathcal{A}_1 \otimes \mathcal{A}_1) / S $
and where $S$ is the ideal generated by 
\[ 
\{ \hspace{0.05cm}a \otimes b- b \otimes a,\, a \centerdot\alpha \otimes b - a \otimes b \centerdot\alpha 
\hspace{0,1cm} \hspace{0.1cm} | \hspace{0.1cm}
\hspace{0,1cm} a,b \in \mathcal{A}_1, \alpha \in \mathcal{A}_0 \hspace{0.05cm} \}. 
\]
If we denote by $a \odot b$ the image of $a \otimes b$ in $\mathcal{G}_0$, one can write the Lie (super) bracket : 
$$
\begin{array}{rcl}
[a,b] & =& a \odot b \\[6pt]
[a \odot b,c] & =& a \centerdot (b\centerdot c) + b \centerdot (a \centerdot c) \\[6pt]
[a \odot b,c \odot d] &=& 2a \centerdot (b \centerdot c) \odot d + 2b \centerdot (a \centerdot d) \odot c 
\end{array}
$$

There is a natural action of $\mathcal{G}_\mathcal{A}$ on the corresponding
Lie antialgebra~$\mathcal{A}$, so that there is a Lie algebra homomorphism
$$
\mathcal{G}_\mathcal{A}\to{}Der(\mathcal{A}).
$$
Indeed, the action of the odd part $\mathcal{G}_1$ is given by the right multiplication
and this generates the action of $\mathcal{G}_0$, cf \cite{Ovs2011}.
Note that, in the above examples, one has :   
$
\mathcal{G}_{\mathcal{K}_3} \cong \mathrm{osp}(1|2)$ and
$\mathcal{G}_{\mathcal{AK}(1)} \cong \mathcal{K}(1).
$

In general, the adjoint Lie superalgebra is not isomorphic to the Lie superalgebra of derivations.

\subsection{Definition of $\mathcal{J}_{g,N}$}

A new series of Lie antialgebras extended $\mathcal{AK}(1)$ was found by Leidwanger and Morier-Genoud; see \cite{LM2011}.
These algebras are related to Riemann surfaces with marked points and are called the
{\it Jordan superalgebras of Krichever-Novikov type}, $\mathcal{J}_{g,N}$. Remind that a theta characteristics is still fixed,  
the even part of $\mathcal{J}_{g,N}$ is the space of meromorphic functions,
$\mathfrak{a}_{g,N} \cong \cF_{0}$, while the odd part is the space of $-1/2$-densities.

\begin{df}
The Lie antialgebra $\mathcal{J}_{g,N}$ is the vector superspace 
$\mathfrak{a}_{g,N}\oplus \cF_{-1/2}$ 
equipped with the product
\begin{align*}
e(z) \centerdot f(z)  &  =  e(z)  \bullet f(z)  \\
e(z)  \centerdot  \psi(z) (dz)^{-1/2} & = \frac{1}{2}  e(z) \bullet \psi(z) (dz)^{-1/2}   \\
\varphi(z) (dz)^{-1/2} \centerdot  \psi(z) (dz)^{-1/2} & = \{ \varphi(z) (dz)^{-1/2} ,  \psi(z) (dz)^{-1/2} \} .
\end{align*}
\end{df} 
More precisely, we can write: 
\[
\left\{
\begin{array}{l}
	e \centerdot f = ef \\
	e \centerdot \psi \hspace{0.05cm} (dz)^{-1/2} = \frac{1}{2} e  \hspace{0.05cm}  \psi   \hspace{0.1cm} (dz)^{-1/2} \\
	\varphi  \hspace{0.05cm} (dz)^{-1/2} \centerdot \psi  \hspace{0.05cm}  (dz)^{-1/2} = -\frac{1}{2} \varphi'  \hspace{0.05cm} \psi + \frac{1}{2} \varphi   \hspace{0.05cm} \psi'.
\end{array}
\right.
\]
It is shown in \cite{LM2011}, that  the adjoint Lie superalgebra of $\mathcal{J}_{g,N}$
coincides with $\mathcal{L}_{g,N}$.

\begin{ex}
\label{Emx}
a)
In the case of two marked points $ A = \{0 \} \cup \{\infty \}$ on the Riemann sphere, 
the algebra $\mathcal{J}_{0,2}$ can be identified with $\mathcal{AK}(1)$. 

b)
A beautiful example in the case of three punctures on the Riemann sphere is considered in \cite{LM2011}.
One can fix $A = \{- \alpha, \alpha \} \cup \{\infty \},$  where $\alpha  \in \bbC \backslash  \{ 0 \}$.
The Jordan superalgebra $\mathcal{J}_{0,3}$ has the basis 
\[
\begin{array}{ll}
G_{2k}(z) =  (z-\alpha)^k (z+ \alpha)^k , &  
G_{2k+1}(z) = z  (z-\alpha)^k (z+ \alpha)^k,\\[8pt]
\varphi_{2k+\frac{1}{2}}(z) = \sqrt{2} z (z-\alpha)^k (z+ \alpha)^k  \,dz^{-1/2}, & 
 \varphi_{2k-\frac{1}{2}}(z) = \sqrt{2}  (z-\alpha)^k (z+ \alpha)^k \,dz^{-1/2} ,
 \end{array}
\]  
where $k$ $\in \mathbb{Z}.$ Remark that the generators of the odd parts of $\mathcal{L}_{g,N}$ and $\mathcal{J}_{g,N}$ are the same.   
The sub-superalgebra 
 $\mathcal{J}^-_{0,3} = \left\langle{}G_n \hspace{0.02cm} : 
\hspace{0.02cm} n \leq 0 \hspace{0.05cm},\hspace{0.05cm} \varphi_i  \hspace{0.02cm} : 
\hspace{0.02cm}  i \leq \frac{1}{2}\right\rangle$ 
is isomorphic to $\mathcal{AK}(1)$. 
More precisely, the embedding $ \iota : \mathcal{AK}(1) \hookrightarrow \mathcal{J}_{0,3}$
is defined on the generators as follows: 
\[
\begin{array}{lll}
\iota (\varepsilon_{-1}) = 
G_0 + 2 \alpha G_{-1} + 2 \alpha^2 G_{-2}, \hspace{0.5cm} &
 \iota (\varepsilon_{1}) =
 G_0 - 2 \alpha G_{-1} + 2 \alpha^2 G_ {-2}, \hspace{0.5cm}  & 
\iota  (\varepsilon_{0}) = G_0 \\ [6pt] 
\iota (a_{-\frac{1}{2}}) = \frac{1}{2 \sqrt{\alpha}} (\varphi_{1/2} + \alpha \varphi_{- 1/2}), &
\iota (a_{\frac{1}{2}}) = \frac{1}{2 \sqrt{\alpha}} (\varphi_{1/2} - \alpha \varphi_{- 1/2}),&
\end{array}
\]
see \cite{LM2011} for the details.
\end{ex}

\section{1-cocycles with values in the dual space}

In this section, we construct 1-cocycles on $\mathcal{L}_{g,N}$ and $\mathcal{J}_{g,N}$
with values in the dual space.
In the Lie case, existence of such a 1-cocycle is almost equivalent to the existence
of a 2-cocycle with trivial coefficients (\ref{CocLie}).
In the Jordan case, the situation is different.
It was proved in~\cite{Ovs2011} and \cite{LO2011} that the Jordan superalgebra $\mathcal{J}_{g,N}$
has no non-trivial central extensions.
Therefore, there is no 2-cocycle on $\mathcal{J}_{g,N}$ analogous to (\ref{CocLie}).
However, there exists a nice construction of 1-cocycle that has very similar properties.

\subsection{1-cocycle on the K-N Lie superalgebras}

Given a 2-cocycle on a Lie (super) algebra $c:\cL\times\cL\to\mathbb{C}$,
one can define a 1-cocycle, $C$, on~$\cL$ with values in the dual space  $\cL^*$.
The definition is as follow 
\begin{equation}
\label{DefCoc}
\langle C(x),y \rangle := c(x,y),
\end{equation}
for all $x,y\in\cL$.
The $1$-cocycle condition: 
$$
C\left( [x,y]\right) 
= ad^*_{x} (C(y) ) - 
(-1)^{\bar{x}\bar{y}} ad^*_{y}(C(x))
$$
follows from the 2-cocycle condition for $c$.
Note that the converse construction does not work since
 $c$ is not necessarily skewsymmetric. 
 
The 2-cocycle (\ref{CocLie}) defines, therefore, a 1-cocycle on every
Lie superalgebra $\cL_{g,N}$.
 Since a theta characteristics is fixed, the separating cycle $\mathcal{C}$ defined a natural pairing (called a K-N pairing in~\cite{Sch2003}, p58) between $\mathcal{F}_{\lambda}$ and $\mathcal{F}_{1-\lambda}$ which is given by 
\[
\mathcal{F}_{\lambda} \times  \mathcal{F}_{1-\lambda} \longrightarrow \mathbb{C}\hspace{0.3cm} : \hspace{0.2cm}\left(f,g \right) \longmapsto \left< f ,  g \right> := \frac{1}{2 i \pi} \int_{\mathcal{C}}  f \bullet g  .
\]
So that, the space 
\[
\mathcal{F}_{2} \oplus \mathcal{F}_{3/2}
\]
will be seen has a nice geometric sub-space of the dual space $\cL^*_{g,N}$ thanks to the pairing. 

Since for every $a,b$ $\in \mathcal{L}$ and for every $u$ $\in \mathcal{L}^*$ we must have 
\begin{eqnarray}
\left< ad^{*}_{a}u, b \right> := - (-1)^{\bar{a} \bar{u} }\left< u, ad_ {a} b \right>, \label{eq1}
\end{eqnarray} we can see that the coadjoint action of $\cL_{g,N}$ is:
\begin{eqnarray*}
ad^*_{\varphi(z) (dz)^{-1/2}}  \left(u(z) (dz)^2 \oplus w(z) (dz)^{3/2}  \right) 
&=& -  \{ \varphi,   w \}  \oplus -\frac{1}{2} \varphi \bullet u 
\\
& = & - \left( \frac{3}{2}\varphi'  w 
+ \frac{1}{2} \varphi w' \right)  \hspace{0.1cm} (dz)^2 \oplus -\frac{1}{2} \varphi u \hspace{0.1cm} (dz)^{3/2} 
\\
ad^*_{e(z) (dz)^{-1}}  \left(u(z) (dz)^2 \oplus w(z) (dz)^{3/2}  \right)  
&=&   \{ V ,  u \}  \oplus  \{ V ,   w \}
\\
&= & \left( 2 e' u + eu' \right) \hspace{0.1cm}  (dz)^2 \oplus \left( \frac{3}{2} e' w  + ew' \right)  \hspace{0.1cm} (dz)^{3/2} 
\end{eqnarray*}
where $u,w, e$ and $\varphi$ are some meromorphic functions on the surface.

\begin{prop}
\label{OnP}
A local 1-cocycle on $\mathcal{L}_{g,N}$ is given by
\begin{equation}
\label{OneCocLie}
\begin{array}{rcl}
C\left(e(z) \frac{d}{dz}\right) 
&=&
-  \left(e^{'''}- 2Re'-R'e\right) dz^2,\\[6pt]
C\left(\varphi(z)dz^{-1/2}\right) 
 &=&
\left(  \varphi^{''} 
   -\frac{1}{2}  R \varphi\right) dz^{3/2}
\end{array}
\end{equation}
\end{prop}

\begin{proof}
Straightforward from (\ref{CocLie}). Remark that this cocycle is \emph{local} since we fixed a theta characteristics and a splitting on the Riemann surface.
\end{proof}

\begin{ex}
In the case of Riemann sphere ($g=0$) with respect to the splitting, the
1-cocycle (\ref{DefCoc}) related to (\ref{CocLie}) with $R\equiv0$, reads simply: 
\begin{equation}
\label{OnPR0}
C \left( e(z)\frac{d}{dz} \right) =  - e'''(z)\, dz^{2}, 
\qquad C \left( \psi(z) \frac{d}{dz^{1/2}} \right) = \psi''  (z)\, dz^{3/2}, 
\end{equation}
where $z$ is the standard coordinate.
\end{ex}

\subsection{1-cocycle on $\mathcal{L}_{0,3}$}

In the case of the Lie superalgebra $\mathcal{L}_{0,3}$
(and further with the Jordan superalgebra $\mathcal{J}_{0,3}$) the constructed 1-cocycle can be
 calculated explicitly. 

The space  $\cL^*_{0,3}$ has the following basis:
\begin{align*} 
\varphi^*_{2k - 1/2} & 
=  \frac{1}{\sqrt{2}} z (z^2 - \alpha^2)^{-k-1} (dz)^{3/2}, &
  V^*_{2k} & 
=  (z^2 - \alpha^2)^{-k-1} (dz)^2,  \\
\varphi^*_{2k + 1/2} 
& = \frac{1}{\sqrt{2}}  (z^2 - \alpha^2)^{-k-1} (dz)^{3/2}, &
   V^*_{2k+1}  & =  z (z^2 - \alpha^2)^{-k-2} (dz)^2,
\end{align*}
dual to (\ref{Sch}) and (\ref{SchBis}).

\begin{prop}
Up to isomorphism, the 1-cocycle on the Lie superalgebra $\mathcal{L}_{0,3}$ related to (\ref{OnPR0}) that vanishes on $osp(1|2)$ is given by: 
\begin{align*}
C(V_n)  = 
&  \hspace{0.2cm}   -n ( n-1 ) (n+1)  V^*_{-n} 
- 2 \alpha^2 n (n-2) ( n- 1  ) V^*_{-n+2}  
- \alpha^4 n (n-2) (n-4) V^*_{-n+4}    \\[6pt]
C(V_m)  = 
&  \hspace{0.2cm} - (m+1) m (m-1) V^*_{-m} 
-  \alpha^2  (m+1) (m-1) (m-3)V^*_{-m+2} ,  \\[6pt]
C(\varphi_i) = 
& \textstyle
 \hspace{0.2cm} 2 (i + \frac{1}{2})(i - \frac{1}{2}) \varphi^*_{-i} 
+ 2 \alpha^2 (i-\frac{1}{2}) (i-\frac{5}{2}) \varphi^*_{-i+2} ,  \\[6pt]
C(\varphi_j)  = 
&\textstyle
 \hspace{0.2cm} 2 (j + \frac{1}{2})(j - \frac{1}{2}) \varphi^*_{-j} 
+ 2 \alpha^2 (j + \frac{1}{2}) (j - \frac{3}{2}) \varphi^*_{-j+2},  
\end{align*}
where $i- \frac{1}{2}$, $n$ are even and $ j - \frac{1}{2}$, $m$ are odd.
\end{prop}

\begin{proof}
This is a simple application of the general formulas (\ref{OneCocLie}) with $R\equiv0$.
\end{proof}

\subsection{Modules and cocycles on Lie antialgebras}

Cohomology of Lie antialgebras was studied in \cite{LO2011}. 
In particular, a (unique up to constant) 1-cocycle $C:\mathcal{AK}(1) \longrightarrow \mathcal{AK}(1)^* $ 
vanishing on  $\mathcal{K}_3$ was constructed. 
 
Let us recall several basic notions from \cite{LO2011}. Let $\cB$ be a $\bbZ_2$-graded vector space and 
$\rho : \mathcal{A} \rightarrow End(\cB)$ an even linear function.
If $\mathcal{A} \oplus \cB$ equipped with the product 
\[
(a,b) \centerdot (a',b') = \left( a \centerdot a', \rho_{a}(b') + (-1)^{\bar{a'}\bar{b}} \rho_{a'} (b) \right)  
\]
for all homogeneous elements $a,a' \in \mathcal{A}$ and $b, b' \in \cB$, is 
 a Lie antialgebra then ($\cB,\rho)$ is called an \emph{$\mathcal{A}$-module}. 
This structure is called a \emph{semi-direct sum} and denoted by $\mathcal{A} \ltimes \cB .$ 
Given an $\mathcal{A}$-module $\cB$, the dual space $\cB^*$ 
is also an $\mathcal{A}$-module, the $\mathcal{A}$-action being given by 
\begin{eqnarray}
\left< \rho^{*}_{a}u, b \right> := (-1)^{\bar{a} \bar{u} }\left< u, \rho_ {a} b \right>, \label{eq1}
\end{eqnarray}
for all homogeneous elements $a \in \mathcal{A}$,  $b \in \cB$ and $ u \in \cB^*$.

A \emph{$1$-cocycle} on a Lie antialgebra $\mathcal{A}$ with coefficients in an 
$\mathcal{A}$-module $\mathcal{B}$, is an even linear map 
$\mathcal{C} : \mathcal{A}\longrightarrow \mathcal{B}$ such that 
\begin{align}
\mathcal{C}\left( u \centerdot v\right) = 
\rho_{u} \left( \mathcal{C}(v) \right) +  (-1)^{\bar{u}\bar{v}} \rho_{v}\left( \mathcal{C}(u) \right). 
\label{eqn1cocycleantiLie}
\end{align}
A Lie antialgebra is tautologically a module over itself,
the adjoint action $ad : \mathcal{A} \longrightarrow End ( \mathcal{A})$ 
defined such that $ ad_{a}(a')  =a \centerdot a'$ for all $ a, a' \in \mathcal{A}$. 
So that, the dual space, $\mathcal{A}^*$, is an 
$\mathcal{A}$-module as well.

\subsection{1-cocycles on $\mathcal{J}_{g,N}$}

It was already proved in \cite{Ovs2011} that a Lie antialgebra
has no non-trivial central extensions, provided the even
part contains a unit element. It follows that the algebras $\mathcal{J}_{g,N}$ have no non-trivial 2-cocycles.
However, one has the following :

\begin{thm}
\label{MainSecond}
(i)
With respect to the splitting, the expression 
\begin{equation}
\label{ACoc}
\mathcal{C}\left( \varepsilon (z) \right) =  - \varepsilon ' (z)dz, 
\qquad \mathcal{C} \left( \psi (z) dz^{-1/2} \right) = 
\left( \psi''(z)- \frac{1}{2}R\,\psi (z)\right) dz^{3/2}
\end{equation}
defines a local 1-cocycle on $\mathcal{J}_{g,N}$ with coefficients in $\mathcal{J}^*_{g,N}$.

(ii)
Consider the cocycle  (\ref{ACoc}) with $R\equiv0$ and vanishing on the subalgebra $\mathcal{K}_3$, then if $N=2$ it is unique up to a multiplicative constant and if $N=3$ it is unique up to isomorphism. 

\end{thm}

\begin{proof}
Part (i).
Similarly to formula (\ref{CocLie}),
the expression in the right-hand-side of (\ref{ACoc}) is independent of the choice of the
coordinate $z$.
One now easily checks  that this expression indeed satisfies 
the condition (\ref{eqn1cocycleantiLie}) of $1$-cocycle. 
This follows from the relations : 
\begin{eqnarray*}
ad^*_{\varphi(z) (dz)^{-1/2}}  \left(u(z) dz \hspace{0.1cm} \oplus \hspace{0.1cm} w(z) (dz)^{3/2}  \right) & =&
\textstyle
  - \frac{1}{2} \varphi w  \hspace{0.1cm}  dz \oplus - \left( \frac{1}{2} \varphi u' 
  + \varphi' u \right)  \hspace{0.1cm}   (dz)^{3/2} \\[4pt]
ad^*_{\varepsilon(z)}  \left( u(z) dz \oplus w(z) (dz)^{3/2}  \right) & = &
\textstyle
\varepsilon  u \hspace{0.1cm}dz\hspace{0.1cm} \oplus \hspace{0.1cm}\frac{1}{2} \varepsilon w \hspace{0.1cm}   (dz)^{3/2} 
\end{eqnarray*}
where $u,w, \varepsilon$ and $\varphi$ are some meromorphic functions on the surface.

Part (ii).
Let us first consider the case $N=2$ and show that the cocycle 
(\ref{ACoc}) with $R\equiv0$
from $\mathcal{AK}(1)$ to $\mathcal{AK}(1)^*$ is the unique (up to a multiplicative constant) 1-cocycle
that vanishes on $\mathcal{K}_3$.
The explicit formula in the basis is  given in \cite{LO2011}, for all $n \in \mathbb{Z}$ and all $i \in \mathbb{Z}+ \frac{1}{2}$  :
\begin{equation}
\label{LOCoc}
\textstyle
\mathcal{C}(\varepsilon_{n}) = 
-n \varepsilon^*_{-n} , \hspace{0.5cm} 
\mathcal{C}(a_i )= \left( i-\frac{1}{2} \right) \left( i+ \frac{1}{2} \right) a^*_{-i}.
\end{equation}
Assume that $\mathcal{C} : \mathcal{AK}(1) \longrightarrow \mathcal{AK}(1)^*$ is a 1-cocycle.
Since $\mathcal{C}$ is even, it is of the form
\[
\mathcal{C} \left( \varepsilon_n \oplus a_i \right)\hspace{0.2cm} = \hspace{0.2cm} \mathcal{C} \left( \varepsilon_n \right) \oplus \mathcal{C} \left( a_i \right)  \hspace{0.2cm}= \hspace{0.2cm}\sum_{r \in \mathbb{Z}} \hspace{0.2cm} \lambda_n^r \varepsilon^*_r \oplus  \sum_{k \in \mathbb{Z}+ \frac{1}{ 2} } \mu_i^k  a^*_k .
\] 
The condition of 1-cocycle (\ref{eqn1cocycleantiLie}) gives : 
\begin{align*}
\mathcal{C} \left( \varepsilon_n \centerdot \varepsilon_m \right) 
&= \hspace{0.1cm}  ad^*_{\varepsilon_n} \mathcal{C} \left( \varepsilon_m \right) +
ad^*_{\varepsilon_m} \mathcal{C} \left( \varepsilon_n \right)  & \Leftrightarrow & & \lambda_{n+m}^{r} = \hspace{0.1cm}  & \lambda_{m}^{r+n} + \lambda_{n}^{r+m} \\
\mathcal{C} \left( \varepsilon_n \centerdot a_i \right) &=  \hspace{0.1cm}  
ad^*_{\varepsilon_n} \mathcal{C} \left( a_i \right) + ad^*_{a_i} \mathcal{C} \left( \varepsilon_n \right)  &
\Leftrightarrow&&   \mu_{i+n}^{k} = \hspace{0.1cm} 
& \mu_{i}^{k+n} +(k-i ) \lambda_{n}^{i+k} \\
\mathcal{C} \left( a_i \centerdot a_j \right) 
&=   \hspace{0.1cm}  ad^*_{a_i} \mathcal{C} \left( a_j \right) - ad^*_{a_j} \mathcal{C} \left( a_i \right) &\Leftrightarrow&&  (j-i ) \lambda_{i+j}^{r}  = \hspace{0.1cm}  &  - \mu_{j}^{r+i} +  \mu_{i}^{r+j} ,
\end{align*}
for all $n,m,r \in \mathbb{Z}$ and all $i,j,k \in \mathbb{Z}+ \frac{1}{2}.$ 
Since this cocycle vanishes on the Lie antialgebra $\mathcal{K}_3 $
 generated by $ \langle\varepsilon_0, a_{-1/2}, a_{1/2}\rangle$, by induction one has the following
 (unique) solution: 
\[
\textstyle
\lambda_n^r = - n \delta_{r,-n} \hspace{0.3cm} \mbox{and}    
\hspace{0.3cm}   \mu_i^k =
 (k^2-\frac{1}{4}) \delta_{k,-i}    \hspace{0.5cm} \forall  n,r \in \mathbb{Z} ;  
 \hspace{0.2cm}\forall i,k \in \mathbb{Z}+ \frac{1}{2},
\]
and thus obtains the cocycle (\ref{LOCoc}). 

Now, let us show the uniqueness for $N=3$. As proved in \cite{LM2011}, the subalgebra $\mathcal{J}^-_{0,3} =  \left\langle{}G_n : n \leq 0; \hspace{0.1cm}\varphi_i : i \leq 1/2 \right\rangle$ is isomorphic to $\mathcal{AK}(1)$, cf. Example \ref{Emx} b). Suppose that we have a 1-cocycle $\mathcal{C} : \mathcal{J}_{0,3} \longrightarrow \mathcal{J}_{0,3}^*$ and writing it with the elements of the basis as the same way than in the first part of the proof $(ii)$. Using the 1-cocycle condition  (\ref{eqn1cocycleantiLie}), we can can show that if we know the 1-cocycle $\mathcal{C}$ on $\mathcal{J}^-_{0,3}$ (i.e. on $\mathcal{AK}(1)$), then the cocycle is uniquely determined on $\mathcal{J}_{0,3}$ entirely. Hence the result on $\mathcal{J}_{0,3}$ since the 1-cocycle on $\mathcal{AK}(1)$ is unique when it vanishes on $\mathcal{K}_3$. 

\end{proof}


\begin{rem}
The 1-cocycle (\ref{ACoc}) has a very simple and, geometrically, very natural form :
this is the De Rham differential of a function combined with the Sturm-Liouville equation associated to
a projective connection, applied to a $-1/2$-density.
\end{rem}

\subsection{An explicit formula of the 1-cocycle on $\mathcal{J}_{0,3}$}

We finish the paper with an explicit formula of the 1-cocycle (\ref{ACoc}) in the case of 3 marked points.

\begin{prop} \label{1cocycleantiLie}
Up to isomorphism, the 1-cocycle (\ref{ACoc}) on the algebra $\mathcal{J}_{0,3}$ is given by
\begin{align*}
C(G_n)  = &  \hspace{0.2cm}   - n  G^*_{-n} , \\[4pt]
C(G_m)  = &  \hspace{0.2cm} - m  G^*_{- m} - \alpha^2  (m-1) G^*_{-m+2}  , \\[4pt]
C (\varphi_i) = &  
\textstyle
\hspace{0.2cm} 2 (i + \frac{1}{2})(i - \frac{1}{2}) \varphi^*_{-i} +
 2 \alpha^2 (i-\frac{1}{2}) (i-\frac{5}{2}) \varphi^*_{-i+2}  , \\[4pt]
C (\varphi_j)  = & 
\textstyle
\hspace{0.2cm} 2 (j + \frac{1}{2})(j - \frac{1}{2}) \varphi^*_{-j} + 
2 \alpha^2 (j + \frac{1}{2}) (j - \frac{3}{2}) \varphi^*_{-j+2} ,
\end{align*}
where $i- \frac{1}{2}$ and $n$ are even and $ j - \frac{1}{2}$, $m$ are odd.
\end{prop}

\begin{proof}
Elements of the basis of the dual space $\mathcal{J}^*_{0,3}$ are given by:
\begin{align*} 
\varphi^*_{2k - 1/2} & =  \frac{1}{\sqrt{2}} z (z^2 - \alpha^2)^{-k-1} (dz)^{3/2}, &  G^*_{2k} & = z  (z^2 - \alpha^2)^{-k-1} dz , \\
\varphi^*_{2k + 1/2} & = \frac{1}{\sqrt{2}}  (z^2 - \alpha^2)^{-k-1} (dz)^{3/2}, &   G^*_{2k+1}  & =   (z^2 - \alpha^2)^{-k-1} dz .
\end{align*}
The computations are straightforward.
\end{proof}

\medskip

\noindent \textbf{Acknowledgments}.
I am grateful to Pierre Lecomte and Valentin Ovsienko for a number of
enlightening discussions and support. 
I am also pleased to thank
Sophie Morier-Genoud and Severine Leidwanger for their interest and help.   
Furthermore, I would like to thank M. Schlichenmaier for helpful and valuable comments, remarks and suggestions. 

\vskip 1cm











\end{document}